\documentclass[12pt,oneside]{amsart} 
\usepackage{amssymb,amscd}
\usepackage{euscript}
 
      \theoremstyle{plain}
      \newtheorem{theorem}{Theorem}[section]
      \newtheorem{lemma}[theorem]{Lemma}
      
      \newtheorem{proposition}[theorem]{Proposition}
      \newtheorem{remark}[theorem]{Remark}
      
      \newtheorem{definition}[theorem]{Definition}        
          
\numberwithin{equation}{section}

      \makeatletter
      \def\@setcopyright{}
      \def\serieslogo@{}
      \makeatother

\def\A{\EuScript{A}} 
\def\B{\EuScript{B}}

\def\V{\mathcal{V}}
\def\n{\mathcal N}
\def\M{\mathcal M}

\def\Z{\mathbb Z}
\def\N{\mathbb N}
\def\dist{\text{dist}}
\def\diam{\text{diam}}
\def\Id{\text{Id}}
\def\e{\epsilon}

\def\b{\beta}
\def\va{\varphi}

\def\QED{\hfill\hfill{\square}}

\begin{document}

\date{\today}
\author{Boris Kalinin and Victoria Sadovskaya$^{\ast}$}

\address{Department of Mathematics, The Pennsylvania State University, 
University Park, PA 16802, USA.}
\email{kalinin@psu.edu, sadovskaya@psu.edu}

\title[Boundedness, compactness, and invariant norms for  Banach cocycles]
{Boundedness, compactness, and invariant norms \\ for  Banach cocycles over hyperbolic systems} 

\thanks{$^{\ast}$ Supported in part by NSF grant DMS-1301693}


\begin{abstract} 

We consider group-valued cocycles over dynamical systems
with hyperbolic behavior. The base system is either  a hyperbolic diffeomorphism 
or a mixing subshift of finite type.
The cocycle $\A$ takes values in the group of invertible bounded linear 
operators on a Banach space and is H\"older continuous.
We consider the periodic data of $\A$, i.e. the set of its return values along the 
periodic orbits in the base.
We show that if the periodic data of $\A$ is uniformly quasiconformal
or bounded or  contained  in a compact set, then so is the cocycle.
Moreover, in the latter case the cocycle is isometric with respect to a 
H\"older continuous family of norms. We also obtain a general result 
on existence of a measurable family of norms invariant under a cocycle.

\end{abstract}

\maketitle 

\section{Introduction and statement of the results}

Group-valued cocycles appear naturally and play an important role in dynamics. 
In particular, cocycles over hyperbolic systems  have been 
extensively studied starting with the work of A. Liv\v{s}ic \cite{Liv1,Liv2}. 
One of the main problems in this area is to obtain properties of the cocycle 
from its values at the periodic orbits in the base, which are abundant for 
hyperbolic systems. The study encompassed various types of groups,
from  abelian to compact non-abelian and more general 
non-abelian, see  \cite{NT95,PP,Pa,Sch,PW,LW,K11,KS10,S15,G,KS16} 
and a  survey in \cite{KtN}. Cocycles with values in the group of invertible linear operators on a 
vector space $V$ are the prime examples in the last class. The case of  finite dimensional $V$ 
has been well studied, with various applications including derivative cocycles 
of smooth dynamical systems and random matrices. The infinite dimensional case is more 
difficult and is less developed so far. The simplest examples are given by random and Markov sequences of operators. In our setting they correspond to locally constant cocycles over 
subshifts of finite type. Similarly to finite dimensional case, the derivative of a smooth 
infinite dimensional system gives a natural example of an operator valued cocycle. 
We refer to monograph \cite{LL} for an overview of results in this area and to  \cite{GK,M} 
for some of the recent developments.

In this paper we consider cocycles of invertible bounded operators on a Banach space $V$
over dynamical systems with hyperbolic behavior.  The space
$L(V)$  of bounded linear operators on $V$  is a Banach space equipped 
with the operator norm $\|A\|=\sup \,\{ \|Av\| : \,v\in V , \;\|v\| \le 1\}.$ 
The open set $GL(V)$  of invertible elements in  $L(V)$ is a topological group  
and a complete metric space with respect  to the metric
$$
d (A, B) = \| A  - B \|  + \| A^{-1}  - B^{-1} \|.
$$

\begin{definition} Let $f$ be a homeomorphism of a metric space $X$
and let $A$ be a function from $X$ to $(GL(V),d)$. 
The {\em Banach  cocycle over $f$ generated by }$A$ 
is the map $\A:\,X \times \Z \,\to G$ defined  by $\,\A(x,0)=\Id\,$ and for $n\in \N$
 $$
\A(x,n)=\A_x^n = A(f^{n-1} x)\circ \cdots \circ A(x) \quad\text{and}\quad\,
\A(x,-n)=\A_x^{-n}= (\A_{f^{-n} x}^n)^{-1} .
$$
\end{definition}

\noindent Clearly, $\A$ satisfies the {\em cocycle equation}\,
$\A^{n+k}_x= \A^n_{f^k x} \circ \A^k_x$.
\vskip.05cm
Cocycles can be considered in any regularity, but H\"older continuity is the most natural 
in our setting. On the one hand continuity of the cocycle is not sufficient for development 
of a meaningful theory even for scalar cocycles over hyperbolic systems. On the other hand,
symbolic systems  have a natural H\"older structure but lack a smooth one. Moreover, even 
for smooth hyperbolic systems higher regularity is rare for many usual examples of cocycles, 
such as restrictions of the differential to the stable and unstable subbundles. 
We say that a cocycle 
$\A$ is  {\em $\beta$-H\"older} if its generator $A$  is H\"older continuous with exponent 
$0<\beta \le 1$, i.e. there exists $c>0$ such that
$$
 d (A(x), A(y)) \le  \, c\, \dist (x,y)^\beta \quad\text{for all  }x,y \in X. 
$$

For a cocycle $\A$,  we consider the periodic data set $\A_P$  and the set of all values $\A_X$,
$$
\A_P=\{ \A_p^k:\; p=f^kp,\; p\in X,\; k\in \N \} \quad\text{and}\quad
\A_X=\{ \A_x^n:\;  x\in X,\; n\in \Z \}.
$$
Our main result is that uniform quasiconformality, uniform boundedness, 
and pre-compactness of the cocycle can be detected from its periodic data. 
 Moreover, pre-compactness implies that the cocycle is isometric with respect 
 to a H\"older continuous family of norms.

\begin{definition}
The {\em quasiconformal distortion} of a cocycle $\A$ is the function
 $$
Q_\A(x,n)= \| \A_x^n\| \cdot \| (\A_x^n)^{-1}\|, 
 \quad x\in X \text{ and }n\in \Z.
 $$ 
  \end{definition}

\begin{theorem} \label{bounded} 
Let $(X,f)$ be  either a transitive Anosov diffeomorphism of a compact 
connected manifold or a topologically mixing  diffeomorphism of a locally maximal 
hyperbolic set or a mixing subshift of finite type (see Section  2 for definitions). 
Let $\A$ be a 
H\"older continuous Banach 
cocycle over $f$. 
\vskip.1cm
\begin{itemize}
\item[{\bf (i)}] If there exists a constant $C_{per}$ such that 
$ \,Q_\A(p,k)   \le C_{per}\,$ whenever $f^k p=p,$  
 then  $\A$ is {\em uniformly quasiconformal,} 
i.e. there exist a constant $C$  such that 
\vskip.1cm  
\hskip2cm $Q_\A(x,n)\le C\;$ for all $x\in X$ and $n\in Z$.
\vskip.1cm 

\item[{\bf (ii)}] If the  set $\A_P$ is bounded in $(GL(V),d)$, 
then so is the set $\A_X$.
\vskip.1cm
 
\item[{\bf (iii)}]   If the  set $\A_P$ has compact closure in $(GL(V),d)$, 
then so does the set  $\,\A_X$.
\vskip.1cm

\item[{\bf (iv)}]   If the  set $\A_X$ has compact closure in $(GL(V),d)$
then  there exists a H\"older continuous family of norms $\|.\|_x$ on $V$ such that 
\vskip.1cm
\hskip1.2cm $\A_x:(V, \|.\|_x)\to (V, \|.\|_{fx}) \;$ is an isometry  for each $x\in X$.
\end{itemize}
\end{theorem}

We note that the closures in (iii) are not the same in general. For example, 
 if  $\A$ is a {\em coboundary}, i.e. is generated by
$A(x)=C(fx)\circ C(x)^{-1}$ for a function $C:X\to GL(V)$, then $\A_P=\{\Id\}$
while $\,\A_X$ is usually not. The question whether $\A_P=\{\Id\}$ characterizes 
coboundaries has been studied over several decades and answered positively   
for various groups in  \cite{Liv2, NT95,PW, K11, G}.

\begin{remark}
We can view the cocycle $\A$ as an automorphism of the trivial vector bundle 
$\V= X\times V$ which covers $f$ in the base and has fiber maps $\A_x:\V_x \to \V_{fx}$.
Theorem \ref{bounded}  holds in the more general setting where  $X\times V$ is replaced
by a H\"older continuous vector bundle $\V$ over $X$ with fiber $V$ and the cocycle 
$\A$ is replaced by an automorphism $\mathcal F: \V \to \V$ covering $f$. 
This setting is described in detail in Section 2.2 of \,\cite{KS13} and our proofs 
work without any significant modifications.
\end{remark}

Theorem \ref{bounded} extends results for finite dimensional $V$ in \cite{KS10,LW,K11}. 
The infinite dimensional case is substantially different. The initial step 
of obtaining fiber-bunching of the cocycle from its periodic data relies on 
our new approximation results \cite{KS16}. 
The finite dimensional boundedness result is extended in two directions: 
boundedness and pre-compactness, as the latter does not follow automatically.
 Existence of a continuous family of norms requires a new approach. 
 Indeed, on a finite dimensional  space the set of Euclidean norms
has a structure of a symmetric space of nonpositive curvature which was used
in the arguments, 
but in infinite dimensional case there is no analogous metric structure. 
We consider a natural distance on the set of norms 
but the resulting space is not separable so we  work  with a small subset.
The following general result yields a measurable invariant family 
of norms and then we show its continuity.

\begin{proposition} \label{measurable}
Let $f$ be a homeomorphism of a metric space $X$
and let $\A$ be a continuous Banach cocycle over $f$.
If the set of values  $\A_X$ has compact closure in $GL(V)$,
then  there exists a bounded Borel measurable family of norms $\|.\|_x$ on $V$ such that 
$\,\A_x:(V, \|.\|_x)\to (V, \|.\|_{fx}) \,$ is an isometry  for each $x\in X$.
\end{proposition}

Theorem \ref{bounded} yields cocycles with a ``small" set of values $\A_X$, which 
are relatively well understood. 
For example, a cocycle satisfying the conclusion (ii) of the theorem has
 {\em bounded distortion} in the sense of \cite{Sch}, i.e. there exists a 
 constant $c$ such that 
 $$
 d(AB_1,AB_2) \le c \,d(B_1,B_2)\quad\text{and}\quad d(B_1A,B_2A) \le c \,d(B_1,B_2)
 $$ 
 for all $A \in \A_X$ and all $B_1,B_2\in GL(V)$. 
 Some definitive results on cohomology of such cocycles were obtained by K. Schmidt
 in \cite{Sch}. These results can be extended to cocycles satisfying the conclusion of (i)
 by considering the quotient by the group of scalar operators.


\section{Systems in the base} \label{settings}

\noindent{\bf Transitive Anosov diffeomorphisms.} 
Let $X$ be a compact connected  manifold.
A diffeomorphism  $f$ of  $X$
 is called {\it Anosov}\, if there exist a splitting 
of the tangent bundle $TX$ into a direct sum of two $Df$-invariant 
continuous subbundles $E^s$ and $E^u$,  a Riemannian 
metric on $X$, and  continuous  
functions $\nu$ and $\hat\nu$  such that 
\begin{equation}\label{Anosov def}
\|Df_x(v^s)\| < \nu(x) < 1 < \hat\nu(x) <\|Df_x(v^u)\|
\end{equation}
for any $x \in X$ and unit vectors  
$\,v^s\in E^s(x)$ and $\,v^u\in E^u(x)$.
The subbundles $E^s$ and $E^u$ are called stable and unstable. 
They are tangent to the stable and unstable foliations 
$W^s$ and $W^u$ respectively (see, for example \cite{KH}).
Using \eqref{Anosov def} we choose a small positive number
$\rho$ such 
 that for every $x \in \M$ we have
 $\| Df_y\| < \nu (x)$ for all $y$ in the ball  in $W^{s}(x)$ centered 
at $x$ of  radius $\rho$  in the intrinsic metric of $W^{s}(x)$.
We refer to this ball as the local stable manifold of $x$ and denote it by $W^{s}_{loc}(x)$.
Local unstable manifolds 
are defined similarly. It follows that for all $n\in \N$ and $x\in X$,
$$
\begin{aligned}
&\dist (f^nx, f^ny)< \nu^n_x \cdot  \dist(x,y) \quad \text{for all } y\in W^s_{loc}(x),\\
&\dist (f^{-n}x, f^{-n}y)< \hat \nu^{-n}_x \cdot  \dist(x,y) \quad \text{for all }y\in W^u_{loc}(x),
\end{aligned}
$$
where $\,\nu^n_x=\nu(f^{n-1}x)\cdots\nu(x) \,\text{ and }\,
\hat\nu^{-n}_x=(\hat \nu(f^{-n}x))^{-1}\cdots  (\hat\nu(f^{-1}x))^{-1}.$
We also assume that $\rho$ is sufficiently small so that  
$\,W^{s}_{loc}(x) \cap W^{u}_{loc}(z)\,$ consists of a single point
for any sufficiently close $x$ and $z$ in $X$. 
This property is called {\em local product structure}.

A diffeomorphism is said to be {\it (topologically) transitive} if there is a point $x$ in $X$
with dense orbit. All known examples of Anosov diffeomorphisms have this property.
\vskip.2cm


\noindent{\bf Mixing diffeomorphisms of locally maximal hyperbolic sets.} (See Section 6.4 in \cite{KH} for more details.) 
More generally, let $f$ be a diffeomorphism of a manifold $\M$.
A compact $f$-invariant  set $X \subset \M$ is
called {\em hyperbolic} if there  exist a continuous $Df$-invariant splitting 
$T_X \M = E^s\oplus E^u$, and a Riemannian metric and 
continuous functions $\nu$, $\hat \nu$ on an open set
$U \supset X$ such that \eqref{Anosov def} holds for all $x \in X$.
Local stable and unstable manifolds are defined similarly for any $x \in X$
 and we denote  their intersections with $X$ by $W^{s}_{loc}(x)$ and  $W^{y}_{loc}(y)$.
The set $X$ is called {\em locally maximal} if 
$X= \bigcap_{n\in \Z} f^{-n }(U)$ for some open set $U\supset X$. 
This property ensures that $W^{s}_{loc}(x) \cap W^{y}_{loc}(y)$ exists in $X$,
so that $X$ has local product structure.
The map $f|_X$ is called {\em topologically mixing}\,
 if for any two open non-empty subsets $U,V$ of $X$
 there is $N\in \N$ such that $\, f^n(U)\cap V\ne \emptyset\,$ for all $n\ge N$.
 
 In the case of $X=\M$ this gives an Anosov diffeomorphism. 
It is known that mixing holds automatically for 
transitive Anosov diffeomorphisms  of connected manifolds. 


\vskip.2cm
\noindent{\bf Mixing subshifts of finite type.}
Let $M$ be $k \times k$ matrix with entries from $\{ 0,1 \} $ such that all 
entries of $M^N$ are positive for some $N$. Let
$$
X= \{ \,x=(x_n) _{n\in \Z}\, : \,\; 1\le x_n\le k \;\text{ and }\;
 M_{x_n,x_{n+1}}=1 \,\text{ for every } n\in \Z \,\}.
$$
\noindent The shift map $f:X\to X\,$ is defined by 
$(fx)_n=x_{n+1}$.
The system $(X,f)$ is called a {\em  mixing  subshift of finite type}. 
We fix $\nu \in (0,1)$ and consider the metric 
$$
\dist(x,y) = d_\nu(x,y)=\nu^{n(x,y)},
\;\text{ where }\;n(x,y)=\min\,\{ \,|i|\,: \; x_i \ne y_i  \}.
$$
The set $X$ with this metric is compact. The metrics $d_\nu$ for different  
values of $\nu$ are H\"older equivalent.
The following sets play the role of the local stable and unstable 
manifolds of $x$
$$
W^s_{loc}(x)=\{\,y: \;\, x_i=y_i, \;\;i\ge 0\,\}, \quad 
W^u_{loc}(x)=\{\,y:\;\, x_i=y_i, \;\;i\le 0\,\}.
$$
Indeed, for all $x\in X$ and $n\in \N$,
$$
\begin{aligned}
&\dist (f^n x, f^n y )= \nu^n \, \dist  (x,y)  \quad\text{for all } y\in W^{s}_{loc}(x),\\
&\dist  (f^{-n}x, f^{-n}y )= \nu^n\, \dist  (x,y) \quad\text{for all } y\in W^{u}_{loc}(x),
\end{aligned}
$$
and for any $x, z\in X$ with $\dist(x,z) < 1$ the intersection of $W^s_{loc}(x)$ and $W^u_{loc}(z)$
consists of a single point, $y=(y_n)$ such that $y_n=x_n$ for $n\ge 0$
and $y_n=z_n$ for $n\le 0$.


\section{Proofs of Theorem \ref{bounded} and Proposition \ref{measurable}}

\subsection{Fiber bunching and closing property}
First we show that the cocycle $\A$ is {\em fiber bunched}, i.e.
$Q(x,n)$ is dominated by the contraction and expansion in the base in the following sense.

\begin{definition} \label{bunching def}
 A $\beta$-H\"older  cocycle $\A$ 
over a hyperbolic diffeomorphism $f$ is\,
 {\em fiber bunched} if 
there exist numbers $\theta<1$ and $L$  such that for all $x\in X$ and $n\in \N$,
\begin{equation}\label{fiber bunched}
Q_\A(x,n) \cdot  (\nu^n_x)^\beta < L\, \theta^n \quad\text{and}\quad
Q_\A(x,-n) \cdot  (\hat \nu^{-n}_x)^\beta < L\, \theta^n.
\end{equation}
For a subshift of finite type, $\nu(x)=\nu$ and $\hat \nu (x)=1/\nu$, and so the conditions become
 $$
Q_\A(x,n) \cdot  \nu^{\beta |n|} < L\, \theta^{|n|}
\quad \text {for all }n\in \Z.
$$

\end{definition}

Fiber bunching plays an important role in the study of cocycles over hyperbolic systems.
In particular, it ensures certain closeness of the cocycle at the points on the same
stable/unstable manifold.

\begin{proposition} \cite[Proposition 4.2(i)]{KS13}\label{close to Id} 
If $\A$ is  fiber bunched, then there exists $c>0$ 
such that for any $x\in X$ and $y\in W^s_{loc}(x)$,
$$
\|(\A^n_y)^{-1} \circ  \A^n_x - \Id\,\| \leq c\,\dist (x,y)^{\beta}\,
\quad\text{for every }n\in \N, \quad
$$
and similarly for any $x\in X$ and $y\in W^u_{loc}(x)$,
$$
\|(\A^{-n}_y)^{-1} \circ \A^{-n}_x - \Id\,\| \leq c\,\dist (x,y)^{\beta}\,
\quad\text{for every }n\in \N.
$$
\end{proposition}

\noindent This proposition was proven in the finite dimensional case but the 
argument holds for Banach cocycles without modifications.

The base systems that we are considering satisfy the following {\it closing property.}

\begin{lemma} (Anosov Closing Lemma\,  \cite[6.4.15-17]{KH}) \label{Anosov}\,
Let $(X,f)$ be a topologically mixing diffeomorphism 
of a locally maximal hyperbolic set. Then there exist  
constants $D,\, \delta_0 >0$ such that for any $x \in X$ and $k\in\N$ with 
$\dist (x, f^k x) < \delta_0$ there exists a periodic point $p \in X$ with 
$f^k p =p$ such that the orbit segments $x, fx, ... , f^k x$ and 
$p, fp, ... , f^k p\,$ remain close: 
$$
\dist (f^i x, f^i p) \le D  \dist (x, f^k x)\;\,\text{ for every $i=0, ... , k$.}
$$
\end{lemma}

\noindent For subshifts of finite type this property can be observed directly.
Moreover, for the systems we consider there exist $D'>0$ and $0<\gamma <1$ 
such that for the above trajectories
\begin{equation}  \label{d-close-traj}
\dist (f^i x, f^i p) \le D'  \dist (x, f^k x) \, \gamma^{ \min\,\{\,i,\,k-i \,\} }
\quad\text{for every }i=0, ... , k.
\end{equation}
Indeed, the local product structure gives existence of a point 
$y=W^{s}_{loc}(p) \cap W^{u}_{loc}(x)$. Then the contraction/expansion along
stable/unstable manifolds yields the exponential closeness in \eqref{d-close-traj}.

We obtain fiber bunching of the cocycle $\A$ from the following proposition.
Clearly, the assumption in part (i) of the theorem is weaker then the ones in  (i-iv),
and so it suffices to deduce fiber bunching from the assumption in (i).

\begin{proposition} \cite[Corollary 1.6(ii)]{KS16} \label{norm}
Let $f$ be a homeomorphism of a compact metric space $X$
satisfying the closing property \eqref{d-close-traj} 
and let $\A$ be a  H\"older continuous Banach 
cocycle over $f$. If for some numbers $C$ and $s$ we have
$$
Q(p,k) \le Ce^{s k} \quad\text{whenever }p=f^kp,
$$
then for each $\e>0$ there exists a number $C'_\e$ such that
$$
Q(x,n)\le C'_\e \,e^{(s+\e)|n|} \quad\text{for all }x\in X \text{ and }n\in \Z.
$$
\end{proposition}

\noindent We apply the proposition with  $s=0$ and take $\e>0\,$ such that  
$$
e^{\e} \nu^\b <1 \;\text{ and }\; e^{\e} ( \hat \nu^{-1})^\b <1,
\;\text{ where }\;\nu=\max_x \nu(x) \;\text{ and }\;\hat \nu^{-1}=\max_x \hat\nu(x)^{-1}.
$$
Then the fiber bunching condition \eqref{fiber bunched} are satisfied with 
$$
\theta=\max \,\{e^{\e} \nu^\b, \; e^{\e} ( \hat \nu^{-1})^\b \} \quad\text{and}\quad 
L=C'_\e. 
$$

\vskip.1cm 

\subsection{Proof of (i)}
Now we show that  quasiconformal distortion of $\A$ is bounded along a dense orbit.
Since $f$ is  transitive, there is a point $z\in X$ such that its orbit
$$
O(z)=\{ f^n z:\; n\in \Z \} \quad\text{is dense in }X.
$$
We take $\delta_0$ sufficiently small to apply  Anosov Closing 
Lemma \ref{Anosov} and so that  
$$(1+c\delta_0^\beta)/(1-c\delta_0^\beta)\le 2,
\quad\text{where $c$ is as in Proposition \ref{close to Id}.}
$$ 
Let $f^{n_1}z$ and  $f^{n_2}z$ be two points of $O(z)$ with 
$\delta:=\dist(f^{n_1}z, f^{n_2}z)< \delta_0$.
We assume that $n_1<n_2$ and denote 
$$
w=f^{k_1}z \quad\text{and} \quad k=n_2-n_1, \quad \text{so that }\;
\delta =\dist(w, f^k w)< \delta_0. 
$$
Then there exists $p\in X$ with $f^k p =p$ 
such that $\dist(f^i w, f^i p) \le D\delta$ for $i=0, ... , k$.
Let $y$ be the point of intersection of  $W^s_{loc}(p)$ and  $W^u_{loc}(w)$.
We apply Proposition \ref{close to Id} to $p$ and $y$ and to $f^k y$ and $f^k w$ and obtain 
\begin{equation}\label{cl}
\begin{aligned}
 & \|(\A^k_y)^{-1} \circ \A^k_p - \Id \,\| \leq c\delta^{\beta}
\quad\text{and}\quad  \\
& \| \, \A_w^k \circ (\A^k_y)^{-1} - \Id \| = \|(\A^{-k}_{f^k w})^{-1} \circ \A^{-k}_{f^k y} - \Id \,\| 
  \leq c\delta^{\beta}.
 \end{aligned}
\end{equation}

\begin{lemma} \label{distort}
Let $A, B\in GL(V).$ If either $\|A^{-1}B-\Id\,\|\le r $ or $\|AB^{-1}-\Id\,\|\le r$
 for some $r<1$, then 
$$
    (1-r)/(1+r) \le Q(A)/Q(B)\le (1+r)/(1-r),
$$
where $Q(A)=\|A\|\cdot \|A^{-1}\|$ and  $Q(B)=\|B\|\cdot \|B^{-1}\|$
\end{lemma}

\begin{proof} 
Clearly,  $Q(A)=Q(A^{-1})$ and $Q(A_1A_2)\le Q(A_1)Q(A_2)$.
\vskip.1cm

Suppose that $\|A^{-1}B-\Id\,\|\le r $.  We denote $\Delta= A^{-1}B-\Id$. 
Since for any unit vector $v$, $\, 1-r\le \|(\Id+\Delta)v\| \le 1+r$, we have 
$Q(\Id+\Delta) \le (1+r)/(1-r)$.

Since $B= A\,(\Id + \Delta)$ we obtain
   $$
     Q(B)\le Q(A)\cdot Q(\Id+\Delta) \le Q(A)\cdot(1+r)/(1-r).
   $$
Also, $A^{-1}=(\Id+\Delta)B^{-1}$ and hence
   $$
    Q(A)=Q(A^{-1})\le Q(\Id+\Delta) \cdot Q(B^{-1}) \le (1+r)/(1-r)\cdot Q(B),
   $$
and the estimate for $Q(A)/ Q(B)$ follows. The case of $\|AB^{-1}-\Id\,\|\le r$ is similar.
\end{proof}
\vskip.2cm

It follows from the Lemma \ref{distort} and the choice of $\delta_0$ that
$$
   Q(y,k)/Q(p,k) \le (1+c\delta^\beta)/(1-c\delta^\beta)\le 2 
   \quad\text{and}\quad Q(w,k)/Q(y,k) \le 2,
   $$
and hence 
$$
Q(w,k) = Q(f^{n_1} z, \, n_2-n_1) \le 4Q(p,k) \le 4C_{per}.
$$
We take $m\in \N$ such that the set  $\{f^j z;\; |j|\le m\}$ is 
$\delta_0$-dense in $X$. Let
$$Q_m =\max \,\{\,Q(z,j):\; |j|\le m\}.
$$
 Then for any $n>m$
there exists $j$, $|j|\le m$, such that $\dist (f^n z, f^j z)\le \delta_0$
and hence 
  $$
   Q(z,n) \le Q(z,j)\cdot Q(f^j z,\, n-j) \le Q_m \cdot 4C_{per}.
  $$ 
The case of $n<-m$ is similar. 
Thus $Q(z,n)$ is uniformly bounded  in $n \in \Z$, and hence
$Q(f^\ell z,n)$ is uniformly bounded in $\ell,n\in \Z\,$ since
$$
Q(f^\ell z,n) \le Q(f^\ell z, -\ell)\cdot Q(z,n+l)=Q(z, \ell)\cdot Q(z,n+l).
$$
Since $O(z)$ is dense in $X$ and $Q(x,n)$  is continuous on $X$ for each $n$,
this implies that $Q(x,n)$ is uniformly bounded 
in $x \in X$ and $n \in \Z$. 

\vskip.1cm


\subsection{Proof of (ii)}  
Since the set $\A_P$ is bounded,  there is a constant $C'_{per}$ such that 
$$
\max\,\{\|\A^k_p\|,\, \|(\A^k_p)^{-1}\|\} \le C'_{per}\;  \quad\text{whenever $f^k p=p.$}
$$
 We  show that  there exists a constant $C'$ such that 
$$ 
\max\,\{\|\A^n_x\|,\, \|(\A^n_x)^{-1}\|\} \le C' \;\quad\text{for all $x$ and $n$.}
$$
Let $z$, $n_1$, $n_2$, 
$w=f^{n_1}z$,  $k=n_2-n_1$, $y$ and $p$ be  as in (i).
 Since 
 $$(\A^k_y)^{-1} = \left( \Id+ ((\A^k_y)^{-1} \circ \A^k_p - \Id)\right) \circ (\A_p^k)^{-1},
 $$ 
the first inequality in \eqref{cl} implies
$$
\|(\A^k_y)^{-1}\|  \le (1+c\delta^\beta) \cdot \|(\A_p^k)^{-1}\|
 \le (1+c\delta_0^\beta) C'_{per} \le 2 C'_{per}
$$
by the choice of $\delta_0$.
Interchanging $p$ and $y$ we obtain  $\|(\A^k_p)^{-1} \circ \A^k_y - \Id \,\| \leq c\delta^{\beta}$ and 
it follows that $\|\A^k_y \| \le 2 C'_{per}.$
Similarly, the second inequality in \eqref{cl} yields
$$
\|\A^k_w\| \le (1+c\delta^\beta) \cdot \|\A_y^k \| \le 2\|\A_y^k\| \quad \text{and}\quad
\|(\A^k_w)^{-1}\|  \le 2\|(\A_y^k)^{-1}\|,
$$
and we conclude that $\|\A^k_w\| \le 4 C'_{per}$ and $\|(\A^k_w)^{-1}\| \le 4 C'_{per}$.
It follows similarly to (i) that 
$\max\,\{ \|\A^n_x\|, \,\|(\A^n_x)^{-1}\| \}$
is uniformly bounded in $x\in X$ and in $n\in \Z$.
\vskip.1cm


\subsection{Proof of (iii)}
Now we show that  if the set $\A_P$  has compact closure, 
then so does $\A_X$. It suffices to prove that 
$\A_X$  is totally bounded, i.e. for any $\e>0$ it has a finite $\e$-net.
Since $\A_X$ is bounded by (iii), we can choose a constant $M$ such that 
$$
 \| A \|, \,\|A^{-1}\| \le M \;\text{ for all $A\in \A_X$.}  
$$
   
We fix $\,\delta_0>0\,$ sufficiently small to apply Anosov Closing Lemma \ref{Anosov} 
and so that \\
$4M^2c \,\delta_0^\b <\e/2\,$ and  take $\e'$ such that $ \,4M^2c \,\delta_0^\b+M\e' <\e.$
We fix a finite $\e'$-net $P_{\e'}=\{  P_1, ... ,\,P_\ell\}\,$ in $\A_P$. 
  As in (i), we take a point $z$ with dense orbit and
choose $m\in \N$ such that the set $\{f^j z:\; |j|\le m\}\,$  is $\delta_0$-dense in $X$. 
We will show that the set 
$$
\tilde P_\e=\{ \,P_i \circ \A^j_z \,: \;\, i=0,1,..., \ell,  \;\, | j |\le m \},
\quad\text{where $P_0=\Id,$}
$$ 
 is a finite $\e$-net for $\{ \A^n_{z} : \, n \in \Z \}$. 
 Clearly, $\, \A_z^n \in \tilde P_\e\,$ for $\,|n|\le m$.
\vskip.1cm 

Suppose $n>m$, the argument for $n<-m$ is similar. Then 
there exists  $j$ with $|j|\le m$ such that $\delta=\dist (f^j z, f^n z)\le \delta_0$ 
and hence for $x=f^j z$ and $k=n-j$
there is $p= f^kp$ such that
such that $\dist(f^i x, f^i p) \le D\delta$ for $i=0, ... , k$.  
Then it follows from the first inequality in \eqref{cl} that for 
$y=W^s_{loc}(p)\cap W^u_{loc}(x),$
$$
   \| \A^k_p - \A^k_y \|  = \|\A^k_y\circ ((\A^k_y)^{-1} \circ \A^k_p - \Id) \,\| \le
   \|\A^k_y\| \cdot \|(\A^k_y)^{-1} \circ \A^k_p - \Id) \,\|  \le Mc\delta^{\beta},
   $$
 similarly,
   $$
   \| (\A^k_p)^{-1} - (\A^k_y)^{-1} \| =\| ((\A^k_y)^{-1} \circ \A^k_p - \Id)\circ (\A^k_p)^{-1}\|
   \le  Mc\delta^{\beta},
   $$
 and so 
 $$
 d(\A_p^k, \,\A^k_y) =  \| \A^k_p - \A^k_y \| + \| (\A^k_p)^{-1} - (\A^k_y)^{-1} \| \le   2Mc\delta^{\beta}.
 $$
  It follows similarly from the 
 second inequality in \eqref{cl} that $d(\A_y^k,\, \A^k_x) \le  2 Mc\delta^{\beta}$.
 Thus  
$$
d(\A ^k_p, \,\A^k_{f^j z} ) = d(\A ^k_p, \,\A^k_{x} )\le 4Mc \delta^\b =:c' \delta^\b.
$$
Also, there exists an element $P_i$ of the $\e'$-net $P_{\e'}$ such that $d(\A ^k_p, P_i)<\e'$.
So we have $d(\A^k_{f^j z} , P_i)<c' \delta^\b+\e'$.
Then, as $\A^n_z = \A^k_{f^j z} \circ \A^j_z$,  by  Lemma \ref{AB} below we have
 $$
  d(\A^n_z , \, P_i \circ \A^j_z)= d(\A^k_{f^j z} \circ \A^j_z , \, P_i \circ \A^j_z) 
  \le M(c' \delta^\b +\e') \le \e.
  $$ 

\begin{lemma} \label{AB}
If for each of $A,\tilde A, B, \tilde B\in GL(V)$ the norm and 
the norm of the inverse are at most $M$, then
$\,d(A\circ B, \tilde A \circ \tilde  B) = M ( d(A,\tilde A)+d(B,\tilde B) ).$
\end{lemma}
\begin{proof} Adding and subtracting $\tilde A \circ B$ we obtain 
$$
 \| A\circ B- \tilde A \circ \tilde B\| 
 \le \| A - \tilde A \| \cdot \| B\| +  \| \tilde A\| \cdot \| B- \tilde B\|
 \le M (\| A - \tilde A \| +  \| B- \tilde B\|).
$$
Similarly  $\| A^{-1}\circ B^{-1}- \tilde A^{-1} \circ \tilde B^{-1}\| \le 
M (\| A^{-1} - \tilde A^{-1} \| +  \| B^{-1}- \tilde B^{-1}\|)$.
\end{proof}

 We conclude that the set $\tilde P_\e$ is a finite $\e$-net
for the set  $\{ \A^n_{z} : \, n \in \Z \}$ and hence this set is totally bounded.
It follows that so is the set $\{ \A^n_{f^i z} : \,i, n\in \Z \}$. Indeed, if 
$\{\tilde P_1, ... ,\tilde P_N\}$ is an $\e$-net for $\{ \A^n_{z} : \,n\in \Z \}$, then 
$$
\{ \,\tilde P_i \circ (\tilde P_j)^{-1}: \;1\le i,j \le N \, \}
$$
 is a $2M \e$-net  for $\{ \A^n_{f^i z} : \,i, n\in \Z \}$.
This follows from  Lemma \ref{AB}
since 
$$
d(A,B)=d(A^{-1},B^{-1})\quad\text{and}\quad 
\A^k_{f^i z}=\A^{k+i}_z \circ (\A^i_z)^{-1}.
$$

Since the orbit  of $z$ is dense in $X$, 
the set $\{ \A^n_{f^i z} : \,i, n\in \Z \}$ is dense in the set $\A_X$ and
hence this  set is also totally bounded and its closure is compact.

\vskip.2cm


\subsection{Proof of Proposition \ref{measurable}} 
We denote by $N$ the space of all norms $\va$ on $V$ which 
are equivalent to the fixed background norm $\va_0=\|.\|$,  and by $N_K$ the subset 
of the norms equivalent to $\|.\|$ with a constant $K>0$, i.e. 
\begin{equation}\label{equiv}
N_K=\{ \,\va: \;K^{-1}\|v\| \le \va(v) \le K\|v\| \quad\text{for all } v\in V\}.
\end{equation}
We consider the following metric on $N$.
Let $\B_1$ and $\B_2$ be the {\em closed}\,  unit balls in $V$ with respect  to norms 
$\va_1$ and $ \va_2$ in $N$. We define
\begin{equation}\label{dist}
  \dist (\va_1, \va_2) =
  \log \,\min\, \{ \, t\ge 1: \;\B_1\subseteq t\B_2\, \text{ and } \B_2 \subseteq t\B_1 \},
\end{equation}
where $t\B=\{tv:\, v\in \B\}$. It is easy to check that the minimum is attained and
that this is a distance on $N$. Moreover, 
\begin{equation}\label{dist2}
   \dist (\va_1, \va_2) = 
   \log\, \min\,\{ t\ge 1: \;\, t^{-1}\va_2(v)\le \va_1(v) \le t\va_2(v) \,\text{ for each }v\in V\,\},
\end{equation}
and hence  diam$N_K = 2\log K$ as $N_K$ is the closed ball of radius $\log K$ 
centered at $\va_0$.
We note that the space $N_K$ with this metric  is not separable in general.
 
 \begin{lemma} For each $K>0$, the distance on $N_K$ given by \eqref{dist} is equivalent to 
 $$
  \dist ' (\va_1, \va_2) =\sup \,\{ \, |\va_1(v) - \va_2(v)| : \;\|v\| \le 1 \,\},
$$ 
and hence the metric space $(N_K, \dist)$ is complete.
 \end{lemma}
 
 \begin{proof}
 Let $a=\dist (\va_1, \va_2)$.  Then by \eqref{dist2} for any $v$ with  $\|v\|\le1$ we have 
$$ 
\va_2(v) \le e^a \va_1(v) \quad\text{and hence}\quad
\va_2(v)-\va_1(v) \le (e^a-1)\va_1(v) \le  (e^a-1)K,
$$
and similarly, $\,\va_1(v)-\va_2(v) \le  (e^a-1)K.\,$ 
Using the mean value theorem and the fact that $a\le \diam N_K=2 \log K$
we obtain 
$$
   \dist ' (\va_1, \va_2) \le K(e^a-1) \le  Ke^{ 2 \log K} a= K^3\dist (\va_1, \va_2). 
$$
Let $b=\dist' (\va_1, \va_2)$. Then $|\va_1(v) - \va_2(v)| \le b$ for all $v$ with $\|v\|\le 1$,
and hence 
$$ 
|\va_1(v) - \va_2(v)| \le Kb\quad \text{ for all $v$ with $\|v\|\le K$.}
 $$ 
 Suppose that $v\in \B_1$. Then 
 $$
   \va_1(v)\le 1 \;\Rightarrow\; \|v\| \le K \;\Rightarrow\;  \va_2(v)\le \va_1(v) +Kb \le 1+Kb
 $$
and hence $\B_1 \subseteq (1+Kb)\B_2$.  Similarly, $\B_2 \subseteq (1+Kb)\B_1$ and so
\begin{equation}\label{dist comp}
    \dist (\va_1, \va_2) \le \log (1+Kb) \le Kb = K\,\dist' (\va_1, \va_2). 
\end{equation}

So the two metrics on $N_K$ are equivalent.
 It is easy to see that $(N_K, \dist')$ is complete as a closed subset of the complete space
 of bounded continuous functions on the unit ball, and hence $(N_K, \dist)$  is also complete.
  \end{proof}

 Now we construct a Borel measurable family of norms $\va_x=\|.\|_x$ in $N_K$
  such that 
$\A_x:(V, \|.\|_x)\to (V, \|.\|_{fx})$ is an isometry  for each  $x\in X$. 
For a norm $\va\in N$ and an operator $A\in GL(V)$ we denote the pull-back of $\va$ by 
 $A^*\va(v)=\va(Av)$.  The convenience of the metric \eqref{dist}
 is that, as $A(\B_1)\subseteq t A(\B_2)$ if and only if  $\B_1\subseteq t\B_2$, 
the pull-back action of $GL(V)$ on $N$ is isometric, i.e.
 $$
   \dist (A^*\va_1, A^*\va_2) =\dist (\va_1, \va_2) \quad\text{for any }A\in  GL(V)
   \text{ and }\va_1, \va_2 \in N.
 $$
 
\begin{lemma} \label{pull}
For any $\va\in N$ and $A,\tilde A\in GL(V)$ such that $\va,A^*\va,\, \tilde A^*\va\in N_K$,
we have $\,\dist (A^*\va,\, \tilde A^*\va)  \le K^4  \|A-\tilde A\| \,$.
 \end{lemma}
 \begin{proof}
We denote  
$\|A\| _\va=\, \sup \,\{ \,\va(Av):\, \va(v)\le 1\}.$
It follows from \eqref{equiv}  that  
$$\|A\| _\va \le \sup\, \{ K\|Av\|:\, \|v\|\le K\} \le K^2 \|A\|.$$

For any $v$ with $\|v\|\le 1$ we have
$$
  |\va(Av)-\va(\tilde Av)|\le \va(Av-\tilde Av) \le \|A-\tilde A\|_\va \, \va(v)
   \le K^2\|A-\tilde A\|\cdot K \|v\|=K^3\|A-\tilde A\|.
$$
It follows that 
$$
 \dist' (A^*\va,\, \tilde A^*\va) =  \sup \,\{ \, |\,\va(Av)-\va (\tilde Av)\,| : \,\|v\| \le 1 \,\}\,
 \le K^3  \|A-\tilde A\|,
$$
and using \eqref{dist comp} we conclude that 
\vskip.1cm 
\hskip2cm $\dist (A^*\va,\, \tilde A^*\va) \le K \dist' (A^*\va,\, \tilde A^*\va) \le K^4  \|A-\tilde A\|. $
\end{proof}

We denote $\bar \A =\text{Cl} (\A_X)$. 
We fix $K$ such that $K > \|A\|, \|A^{-1}\|$ for all $A\in \bar \A$
and consider the corresponding space of norms $N_K$. 
Then $A^*\va_0 \in N_K$ for each $A\in \bar \A$ as 
$
K^{-1}\|v\| \le \| (A^*\va_0)(v) \|=\|Av\| \le K\|v\|.
$ 
Lemma \ref{pull} implies that the function 
$A \mapsto A^*\va_0$ is continuous in $A$, and since  the set $\bar \A$ is 
compact in $GL(V)$, its image under this function
$$
N_\A=\text{Cl}\,\{\, (\A_x^n)^*\va_0: \;x\in X, \;n\in \Z \,\}\subseteq N_{K} \;\text{ is compact}.
$$
We also note that $A^*\va \in N_{K^2}$  for any $A \in \bar \A$ and $\va \in N_K$.

Since all norms in $N_\A$ are equivalent, the intersection of the unit balls 
$\B_1, ..., \B_n$ for finitely
many of these norms $\va_1, ... , \va_n\in N_\A$ is the unit ball of an equivalent norm 
 $\hat \va=\max \{ \va_1, ... , \va_n \}$ in $N_K$. We consider the set $\hat N_\A$ of all 
 such norms $\hat \va$,
 $$
 \hat N_\A =\{ \,\hat \va=\max \{ \va_1, ... , \va_n \}:\; n\in \N, \; \va_1, ... , \va_n\in N_\A \,\} 
 \subseteq N_K.
  $$
  
  \begin{lemma} The set $\,\bar N_\A = \text{Cl} \,( \hat N_\A)$ is a compact subset of $N_K$.
  \end{lemma}
  \begin{proof}
It suffices to show  that $\hat N_\A$ is  totally bounded. 
Let $P=\{\va_1, ... , \va_N\}$ be an $\e$-net in $N_\A$ and let
 $\hat P$ be the set of all possible maxima of subsets of $P$. Then
 $\hat P$ is an $\e$-net in $\hat N_\A$ since
 \begin{equation}\label{max}
    \dist (\max \{ \va_1, ... , \va_n \}, \max \{ \tilde \va_1, ... , \tilde \va_n \}) \le 
 \max \{ \dist(\va_1, \tilde \va_1), ... ,\dist(\va_n, \tilde \va_n) \}.
\end{equation}
 Indeed, if the right hand side equals $\log t$ then for the corresponding unit balls 
we have 
$\B_1\subset t \tilde \B_1$, ..., $\B_n \subset t \tilde \B_n$ and it follows that
$$
\B_1 \cap ... \cap \B_n \,\subseteq \,t \tilde \B_1 \cap ... \cap t \tilde \B_n 
=\, t (\tilde \B_1 \cap ... \cap  \tilde \B_n).
$$
Similarly, $\,\tilde \B_1 \cap ... \cap  \tilde \B_n \,\subseteq \, t (\B_1 \cap ... \cap \B_n )$,
and so the left hand side of \eqref{max} is at most $\log t$.
\end{proof}
\vskip.1cm

For each $x\in X$ we consider the  pullbacks of the background norm 
by $\A_x^n$ and let
$$
 \va^m_x= \max \,\{\,(\A_x^n)^*\va_0:\; |n|\le m\,\} 
 \quad \text{and} \quad \va_x = \sup \,\{\,(\A_x^n)^*\va_0:\; n \in \Z\,\} .
$$
We note that $\va_x$ is the norm whose unit ball is the intersection
of the unit balls of $\va_x^m$, $m\in \N$, or equivalently the unit balls 
of $(\A_x^n)^*\va_0$, $n\in \Z$. We claim that 
$$
\,\va_x=\lim_{m\to \infty} \va^m_x\; \text{ in }(\bar N_\A, \dist) \;\text{ for each }x\in X.
$$
Indeed,  for each $v\in V$ the sequence $\va_x^m(v)$ increases and converges to $\va_x(v)$.
Since the sequence $\va_x^m$ lies in the compact set $\bar N_\A$, any 
subsequence has a subsequence converging in $(\bar N_\A, \dist)$, whose limit must be $\va_x$.
This implies, by contradiction, that $\va_x^m$ converges to $\va_x$ in $(\bar N_\A, \dist)$.
Note that compactness of $\bar N_\A$ is crucial here.

Since $(\A_x^n)^*\va_0$ depends continuously on $x$ for each $n$, 
the inequality \eqref{max} yields
that $\va^m_x$ is a continuous function on $X$ for each $m$. We conclude that the pointwise
limit $\va_x$ is a Borel measurable function from $X$ to $N_\A \subseteq N_K$. 
By the construction, $\va_x= (\A_x^n)^* \va_{f^nx}$ for all $x\in X$ and $n\in \Z$.
 In other words, $\va_x$ is an invariant section of the bundle $\n=X\times \bar N_\A$ over $X$.


\subsection{Proof of (iv)} 
We keep the notations of  the previous section.
By Proposition \ref{measurable}, there exists a Borel measurable 
 family $\va_x$ of norms in $N_K$ invariant under the cocycle. 
Let $\mu$ be the Bowen-Margulis measure of maximal entropy for $(X,f)$.
We show that the family $\va_x$ coincides $\mu$ almost everywhere 
with a H\"older continuous invariant family of norms. 
First we consider $x\in X$ and $z\in W^s_{loc}(x)$.
We denote $x_n=f^nx$ and $z_n=f^n z$. Since the family of norms $\va_x$
is invariant  and the action of $GL(V)$ on norms is isometric, we have
$$
\begin{aligned}
  \dist(\va_x, \va_z) & = \dist \left( (\A_x^n)^* \va_{x_n},\,(\A_z^n)^* \va_{z_n} \right) \\
    &\le \dist \left( (\A_x^n)^* \va_{x_n},\,(\A_x^n)^* \va_{z_n} \right) +
     \dist \left( (\A_x^n)^* \va_{z_n},\,(\A_y^n)^* \va_{z_n} \right) \\
   &\le \dist (\va_{x_n},\, \va_{z_n}) +  
   \dist \left(  \va_{z_n},\,(\A_z^n \circ (\A_x^n)^{-1}  )^* \va_{z_n} \right). 
     \end{aligned}
$$
By Proposition \ref{close to Id} for all $n\in \N$ we have  
 $\|(\A^n_x)^{-1} \circ  \A^n_z - \Id\,\| \leq  c\,\dist (x,y)^{\beta}$ and hence 
 $$
    \|\A_z^n \circ (\A_x^n)^{-1} -\Id\| \le 
    \|A_z^n\|\cdot \|(\A^n_x)^{-1} \circ  \A^n_z - \Id\,\| \cdot  \|(A_z^n)^{-1}\| \le K^2 c\,\dist (x,z)^{\beta}
 $$
 since $\|\A_x^n\|,\, \|(\A_x^n)^{-1}\| \le K$ for all $x$ and $n$.
Then by Lemma \ref{pull} we have
$$
\dist \left(  \va_{z_n},\,(\A_z^n \circ (\A_x^n)^{-1} )^* \va_{z_n} \right)  \le
 K^8 \|\A_z^n \circ (\A_x^n)^{-1}  - \Id\,\| \leq K^{10} c\,\dist (x,z)^{\beta}
$$
as  $\va_{z_n} \in N_K$ and hence $(\A_z^n \circ (\A_x^n)^{-1} )^* \va_{z_n} \in N_{K^2}$.

Since the the space $(\bar N_\A, \dist)$ is compact and hence separable, 
we can apply Lusin's theorem to the function $\va: x\mapsto \va_x$ from $X$ 
to $\bar N_\A$. So there exists a compact set $S \subset X$ with $\mu(S)>1/2$
on which $\va$ is uniformly continuous. 
Let $Y$ be the set of points in $X$ for which the frequency of visiting $S$ 
equals $\mu(S)>1/2$. By Birkhoff ergodic theorem $\mu(Y)=1$.
If both $x$ and $z$ are in $Y$, then there exists a sequence $\{ n_i \}$ 
such that $x_{n_i}\in S$ and $z_{n_i}\in S$.
Since $z\in W^s_{loc}(x)$, 
$$
\dist(x_{n_i}, z_{n_i})\to 0 \;\text{ and hence }\;
\dist(\va_{x_{n_i}}, \va_{z_{n_i}})\to 0
$$
by uniform continuity of $\va$ on $S$. 
Thus we conclude that for $x,z\in Y$ with $z\in W^s_{loc}(x)$
 $$
   \dist(\va_x, \va_z)\leq  K^{10} c\,\dist (x,z)^{\beta} =: c_1\dist (x,z)^{\beta}.
 $$
Similarly, for $x, y\in Y$ with $y\in W^u_{loc}(x)$ we have 
$\dist(\va_x, \va_y)\leq  c_1 \dist (x,y)^{\beta}$.

We consider a small open set in $X$ with  product structure,
which for the shift case is just a cylinder with a fixed $0$-coordinate.
For almost every local stable leaf, the set of  points of 
$Y$ on the leaf has full conditional measure of $\mu$.  
We consider $x,y \in Y$ which lie on two such local stable leaves and
denote by $H_{x,y}$ the holonomy map along unstable leaves from
$W^s_{loc}(x)$ to $W^s_{loc}(y)$:
$$
  \text{for }z\in W^s_{loc}(x), \quad 
  H_{x,y}(z)= W^u_{loc}(z)\cap W^s_{loc}(y) \, \in W^s_{loc}(y).
$$
It is known that the holonomy maps are absolutely continuous
with respect to the conditional measures of $\mu$, which implies that
 there exists a point 
$z\in W^s_{loc}(x)\cap Y$ close to $x$ such that $z'=H_{x,y}(z)$ is also in $Y$.
By the argument above we have
  $$
   \dist (\va_x,\va_z)\leq  c_1 \dist(x,z)^\beta,  \,
    \dist(\va_z, \va_{z'}) \leq c_1\dist(z,z')^\beta, \,
   \dist(\va_{z'}, \va_y)\leq c_1\dist(z',y)^\beta.
  $$
Since the points $x$, $y$, and $z$ are close,
by the local product structure we have
 $$ 
   \dist(x,z)^\beta + \dist(z,z')^\beta +
    \dist(z',y)^\beta \leq c_2\,\dist(x,y)^\beta.
 $$  

Hence, we obtain $\dist(\va_x,\va_y)\leq  c_3 \,\dist(x,y)^\beta$ 
for  all $x$ and $y$ in a set of full measure $\tilde Y \subset Y$.
We can assume that $\tilde Y$ is invariant by taking
$\bigcap_{n=-\infty}^{\infty} f^n(\tilde Y)$. Since $\mu$ has
full support, the set $\tilde Y$ is dense in $X$, and hence we can extend $\va$ 
from $\tilde Y$ and obtain an invariant H\"older continuous 
family of norms $\,\|.\|_x $ on $X$. 
$\QED$


\end{document}